\documentclass[12pt]{amsart}

\usepackage[utf8]{inputenc}

\usepackage{amsmath,amssymb,amsthm,amstext}
\usepackage{geometry}
\usepackage[
bookmarks=false,
breaklinks=true,
pdfborder={0 0 1},
colorlinks=false
]{hyperref}
\usepackage[numbers]{natbib}

\usepackage{slashed}
\usepackage[all]{xy}
  
\usepackage{accents}   
\numberwithin{equation}{section}

\theoremstyle{plain}
\newtheorem{thm}{Theorem}[section]
\newtheorem{lem}[thm]{Lemma}

\newtheorem{cor}[thm]{Corollary}

\theoremstyle{definition}
\newtheorem{defn}[thm]{Definition}

\theoremstyle{remark}

\theoremstyle{definition}
\newtheorem{ex}[thm]{Example}

\newcommand{\R}{\mathbb R}

\newcommand{\Sym}{\mathrm{Sym}}

\title[Admissibility and Topology]{Admissibility, Topology, and the \\ $\sigma_k$-Loewner--Nirenberg problem}

\author{Hao Fang}
\address{14 MacLean Hall, University of Iowa, Iowa City, IA 52242}
\email{hao-fang@uiowa.edu}

\author{Matthew J. Gursky}
\address{Department of Mathematics, University of Notre Dame}
\email{mgursky@nd.edu}
\begin{document}


\begin{abstract}  The $\sigma_k$-Loewner--Nirenberg problem is a fully nonlinear generalization of the classical Loewner--Nirenberg problem of constructing complete conformal metrics with constant negative scalar curvature in the interior of domains.  For the $\sigma_k$-version, once $k > 1$, one must impose a negative admissibility condition to ensure ellipticity of the equation.  In this note, we exhibit topological obstructions to admissibility and illustrate this with examples.   
\end{abstract}



\maketitle

\section{Introduction}

In this article, we consider singular solutions of the $\sigma_k$-Yamabe problem and show that, under fairly mild conditions, there are topological obstructions to admissibility and hence to the existence of classical solutions.

To explain our results, let $(M,g)$ be a Riemannian manifold of dimension $n \geq 3$, and let $Ric = Ric_g$ and $R = R_g$ denote the Ricci tensor and the scalar curvature of $g$, respectively. The Schouten tensor $A = A_g$ of $g$ is defined by 
\begin{align} \label{Pdef} 
A = \frac{1}{n-2}\left( Ric - \frac{1}{2(n-1)}R g \right). 
\end{align}
We denote by $g^{-1} A$ the $(1,1)$-tensor obtained from $A$ by raising an index. At each point $p \in M$, the endomorphism $g^{-1}A$ of $T_p M$ has $n$ real eigenvalues $(\lambda_1(g^{-1} A),\dots,\lambda_n(g^{-1}A))$.  Let 
\begin{align} \label{skdef}
\sigma_k(g^{-1} A) = \sum_{i_1 < \cdots < i_k} \lambda_{i_1}(g^{-1} A) \cdots \lambda_{i_k}(g^{-1} A)
\end{align}
denote the $k^{\text{th}}$ elementary symmetric function of the eigenvalues of $g^{-1} A$. For example, when $k = 1$,
\[
\sigma_1(g^{-1} A) = \frac{1}{2(n-1)}R.
\]

If $g_u = e^{2u}g$ is a conformal deformation of $g$ and $A_u$ denotes the Schouten tensor with respect to $g_u$, then
\begin{align} \label{Achange}
A_u = A - \nabla^2 u + du \otimes du - \frac{1}{2} |du|^2 g. 
\end{align}
It follows that
\begin{align} \label{Agchange} 
g_u^{-1} A_u = e^{-2u} \, g^{-1} \left( A - \nabla^2 u + du \otimes du - \frac{1}{2} |du|^2 g \right). 
\end{align}

In \cite{jeffduke} Viaclovsky formulated the $\sigma_k$-Yamabe problem: given $(M,g)$ compact, find a conformal metric $g_u$ such that 
\begin{align} \label{skYgen}
\sigma_k\left( g_u^{-1} A_u \right) = \mathrm{constant}. 
\end{align}
Since this equation is fully nonlinear in $u$, some condition on $g_u^{-1}A_u$  is needed to obtain ellipticity. To this end, we introduce the G{\aa}rding cones \cite{Garding1959},
\begin{align} \label{Gammak}
\Gamma_k^{+} = \{ \lambda = (\lambda_1, \dots, \lambda_n) \in \mathbb{R}^n \, : \, \sigma_j(\lambda) > 0 \text{ for all } 1 \leq j \leq k \}. 
\end{align}
These cones are open, convex, and satisfy the nested property
\[
\Gamma_n^{+} \subset \cdots \subset \Gamma_1^{+}.
\]

We say that $u \in C^{\infty}(M)$ is {\em admissible} (or {\em $k$-admissible}) if 
\[
\lambda(g_u^{-1} A_u) = (\lambda_1(g_u^{-1} A_u),\dots,\lambda_n(g_u^{-1} A_u)) \in \Gamma_k^{+}
\]
at each point of $M$. As shown in \cite{jeffduke}, the equation \eqref{skYgen} is elliptic provided $u$ is admissible. It is also elliptic for {\em negatively admissible} functions, i.e.\ for $u \in C^{\infty}(M)$ satisfying 
\[
-\lambda(g_u^{-1} A_u) \in \Gamma_k^{+}.
\]
In this case, we define $\Gamma_k^{-} = -\Gamma_k^{+}$ and write $\lambda(g_u^{-1} A_u) \in \Gamma_k^{-}$.  

More generally, we say that a metric $g$ is admissible (resp.\ negatively admissible) and write $g \in \Gamma_k^{+}(M)$ (resp.\ $g \in \Gamma_k^{-}(M)$) if $\lambda(g^{-1} A_g) \in \Gamma_k^{+}$ (resp.\ $\lambda(g^{-1} A_g) \in \Gamma_k^{-}$) at each point of $M$.

\medskip 

\subsection{The singular Loewner--Nirenberg problem}  We will primarily be interested in solutions of \eqref{skYgen} that are singular on a prescribed set.  To explain our first result, suppose that $(M,g)$ is a compact manifold with boundary. When $k = 1$, the problem of finding a complete conformal metric with constant negative scalar curvature in the interior of $M$ is a special case of the {\em Loewner--Nirenberg problem} (or {\em singular Yamabe problem}), for which solutions are known to exist; see \cite{LN,AM}.  

For $k \geq 2$, the {\em $\sigma_k$-Loewner--Nirenberg problem} seeks a solution $u$ of the problem
\begin{align} \label{skLN} 
\begin{cases} & \sigma_k\left( - g_u^{-1} A_u \right) =  2^{-k} \binom{n}{k} \ \ \mbox{in }M \setminus \partial M, \\
& \lambda(g_u^{-1} A_u) \in \Gamma_k^{-} \ \ \mbox{in } M \setminus \partial M, \\
& u(x) \rightarrow \infty \ \ \mbox{as } x \to \partial M. 
\end{cases}
\end{align} 
We note that the constant on the right is chosen to agree with the value in the model case when $M = B_1(0) \subset \mathbb{R}^n$, and $g_u$ is the standard hyperbolic metric.  

If $M = \Omega \subset \mathbb{R}^n$ is a smooth bounded domain and $g = ds^2$, the Euclidean metric, Gonzalez-Li-Nguyen \cite{GonzalezLiNguyen2018LN} proved the existence and uniqueness of a continuous viscosity solution of (\ref{skLN}).  Subsequently, Li-Nguyen-Wang \cite{LiNguyenWang2018Comparison} showed that viscosity solutions are locally Lipschitz. Li-Nguyen-Xiong \cite{LiNguyenXiong2023Regularity} showed that the solutions $u$ are actually smooth near the boundary, and moreover $u+\log d$ is $C^{n-1,\alpha}$ up to the boundary, where $d$ is the distance to $\partial \Omega$.    

Interestingly, these regularity results are optimal. Li and Nguyen \cite{LiNguyen2021Annuli} showed that the unique viscosity solution of (\ref{skLN}) on an annulus $A \subset \mathbb{R}^n$ is not $C^1$.  This was later strengthened in the paper of Li-Nguyen-Xiong cited above:  in fact, as long as $\partial \Omega$ is disconnected, then the solution cannot be in $C^1$.  

More generally, if $M$ is a compact manifold, possibly with boundary, and if $\partial M\subset\Sigma \subset M$ is closed, one can try to solve (\ref{skLN}) on $\Omega = M \setminus \Sigma$, where the solutions blow up near the singular set $\Sigma$, that is, $u(x) \to \infty$ as $x \to \Sigma$.  For $\Omega \subset \mathbb{R}^n$, Gonzalez-Li-Nguyen studied this problem via exhaustion of $\Omega$ by bounded open sets $\{ \Omega_i \}$.  If $u_i$ is the unique viscosity solution on $\Omega_i$, they showed that $u_{\Omega} = \lim_i u_i$ exists and gave criteria for the blow up on $\Sigma.$
 
Many of the preceding results are special to Euclidean space.   If $(M,g)$ is a compact Riemannian manifold with boundary, then, by the work of Duncan and Nguyen (\cite{DuncanNguyen2026Critical}, \cite{DuncanNguyen2025Subcritical}) viscosity solutions are known to exist for $k \leq n/2$.  The case $k > n/2$ remains open in general. 

In this paper, we show that there are topological obstructions to the existence of negatively admissible functions under a mild  nondegeneracy condition, which we now define:  

\begin{defn}\label{nondegenerate}  Let $M$ be a compact manifold, possibly with boundary.  Let \(\partial M\subset\Sigma\subset M\) be a nonempty compact subset, and denote $\Omega = M \setminus \Sigma$.  

We say that \(u\in C^{\infty}(\Omega)\) is {\em  nondegenerate} on $\Omega$ if the following hold: \vskip.1in 

$(i)$ $u(x)\to+\infty \ \mbox{as } x \to \Sigma.$ \vskip.1in 

$(ii)$ There exists \(\varepsilon_0>0\) such that
\[
|\nabla_g u|_g>0 \ \text{on }
\{x\in\Omega:0<d_g(x,\Sigma)<\varepsilon_0\}.
\]
\end{defn}
 
To simplify the exposition, we state our main result in the special case where the singular set $\Sigma = \partial M$. A much more general result is stated in Section \ref{AdmissibleSec}; see Theorem \ref{thm:vanishing-negative}.  

\begin{thm}   \label{Thm1}  Let $M$ be a compact, connected manifold with boundary.  Let $g$ be a Riemannian metric on $M$ with $\lambda(g^{-1}A_g)\in \overline{\Gamma_{k}^{+}}$.  Denote $\Omega = M \setminus \partial M.$

Suppose \(u\in C^{\infty}(\Omega)\) satisfies \vskip.1in 

$(i)$  $\lambda(g_u^{-1} A_u) \in \Gamma_k^{-}$ at each point, and \vskip.1in 

$(ii)$ $u$ is  nondegenerate on $\Omega$.  \vskip.2in 

Then
\[
H_{n-k+1}(\Omega;\mathbb Z)=\cdots=
H_{n-1}(\Omega;\mathbb Z)=0 .
\]
\end{thm}  

\medskip 

The proof is based on the simple observation that $g^{-1}A_g\in\overline{\Gamma^+_k}$ and $g_u^{-1}A_u \in \Gamma_k^{-}$ at each point 
 impose constraints on the Morse index of any critical point.  The  nondegeneracy condition ensures that all critical points of $u$ are contained in some compact $K \Subset M \setminus \Sigma$.   We remark that if one could prove that solutions were sufficiently regular in a neighborhood of the boundary, as Li-Nguyen-Xiong did for the case of domains in Euclidean space, then the  nondegeneracy assumption could be dropped.   In any case, it is a property that is reasonable to expect (though difficult to establish) for the actual solutions of the PDE.

We illustrate Theorem \ref{Thm1} with two simple examples: 

\medskip 

\begin{ex}
\label{ex:equatorial-complement}
Fix \(p\ge 1\), $n \geq p+2$.  Let $M=\mathbb S^n \setminus \operatorname{int} ( \mathbb S^p \times D^{n-p}),$ where \(\mathbb S^p\times D^{\,n-p}\) is a sufficiently small tubular
neighborhood of a standard equatorial sphere
\(\mathbb S^p\subset\mathbb S^n\).  We equip $M$ with the round metric.  Homotopically, $\Omega = M \setminus \partial M \simeq \mathbb{S}^{n-p-1},$ hence 
\[
H_{n-p-1}(\Omega;\mathbb Z)
\cong
\mathbb Z.
\]
The round metric on \(\mathbb S^n\) belongs to \(\Gamma_k^+\) for every
\(k\). Consequently,
there is no negatively admissible,  nondegenerate
\(u\in C^\infty(\Omega)\) when $k>p+1$. In particular, \(\Omega\) does not admit a
 nondegenerate solution of the \(\sigma_k\)-Loewner--Nirenberg problem once $k > p+1$.  
\end{ex}

\begin{ex}
\label{ex:Sp-products-minus-ball}
For \(p\ge 2\), \(q\ge 2\), let 
\begin{align*}
N = \underbrace{\mathbb S^p \times \cdots \times\mathbb S^p}_{\text{$q$ times}},
\qquad
n=pq,
\end{align*}
and equip \(N\) with the product round metric.  Let $B \subset N$ be a small open geodesic ball, and $M = N \setminus B$.  As before, $\Omega = M \setminus \partial M$.

Since removing a ball does not affect homology below degree \(n-1\),
\[
H_{n-p}(\Omega;\mathbb Z)
\cong
H_{n-p}(N;\mathbb Z).
\]
By the K\"unneth theorem,
\[
H_{n-p}(N;\mathbb Z)
=
H_{p(q-1)}\bigl(\mathbb S^p \times \cdots \times \mathbb{S}^p;\mathbb Z\bigr)
\cong
\mathbb Z^q.
\]
Again, we find that $\Omega$ cannot admit a  nondegenerate, negatively admissible solution of the $\sigma_k$-Loewner--Nirenberg problem if $k>p$.
\end{ex}

\bigskip

Our construction applies more generally when the singular set
\(\Sigma\subset M\) is a closed embedded submanifold of codimension at
least one. From a topological point of view, deleting \(\Sigma\) is equivalent to
deleting a sufficiently small tubular neighborhood of \(\Sigma\).  For example, if
\[
\Sigma=\mathbb S^p\subset\mathbb S^n
\]
is the standard equatorial sphere, then its normal bundle is trivial,
and one may take
\[
N\cong \mathbb S^p\times D^{\,n-p}.
\]
Hence
\[
\mathbb S^n\setminus\mathbb S^p
\simeq
\mathbb S^n\setminus
\operatorname{int}\bigl(\mathbb S^p\times D^{\,n-p}\bigr).
\]

Similarly, deleting a point \(x_0\in M^n\) is homotopy equivalent to
deleting a small open \(n\)-ball:
\[
M\setminus\{x_0\}
\simeq
M\setminus\operatorname{int}D^n.
\]
Therefore the high-codimension examples may be formulated either in
terms of punctured manifolds \(M\setminus\Sigma\) or in terms of compact
manifolds with boundary obtained by deleting tubular neighborhoods of
the singular set.

\medskip

\subsection*{Acknowledgements.}  The second author is partially supported by the Simons Foundation Travel Support
for Mathematicians program, Award ID: SFI-MPS-TSM-00014122.

\section{Topological Results}\label{topology}

In this section we collect the basic Morse-theoretic tools and topological results needed
for the proof of our main results. The key point is that a bound
on the Morse indices of critical points leads to strong restrictions
on the topology of the underlying manifold.

We begin with a standard density result for Morse functions;
see \cite[Proposition~1.2.4]{AD14}.

\begin{thm}
\label{dense}
Let $M$ be a smooth manifold and let $K\subset M$ be compact.
Let $m\in\mathbb N$, and let $f\in C^{\infty}(M)$. Then for every
$\epsilon>0$, there exists a Morse function
$h\in C^{\infty}(M)$ such that
\[
\|f-h\|_{C^{m}(K)}<\epsilon .
\]
\end{thm}

The next two lemmas will be used to obtain an estimate of the Morse index for admissible functions.  The first result can be found in \cite{CNS1985}, Section 1 (see (1.2) and Proposition 1.1):  

\begin{lem} \label{CNSlemma}  If $\lambda=(\lambda_{1},\cdots,\lambda_{n})\in\Gamma_{k}^{+}$,
and $\lambda_{1}\geq\lambda_{2}\geq\cdots\geq\lambda_{n}$, then $\lambda_{k}>0$.
\end{lem}

\begin{lem} 
\label{lem:spectral-convex}
Let $C \subset \R^n$ be an open and convex set invariant under permutations, and define
\[
\mathcal S_C := \{A\in \Sym(n): \lambda(A)\in C\},
\] where \[
\Sym(n) := \{ A \in \R^{n\times n} : A^T = A \}.
\]
Then $\mathcal S_C$ is an open and convex subset of $\Sym(n)$.
\end{lem}
This is a consequence of the Davis convexity theorem 
\cite{Davis1957}; see also \cite{Lewis2000}.

\begin{lem} \label{lem:morse-index-small} Let $(M,g)$ be a Riemannian manifold, and $U \subset M$ an open subset. 
Let $u$ be a smooth Morse function on $U$, and $S$ a symmetric two-tensor on $U$.   Suppose $p\in U$ is a critical point of $u$, and that 
\begin{equation}
    \label{convex-condition}
g^{-1} S \vert_p \in{\Gamma_{k}^{-}},
\qquad
g^{-1} \left( S +\nabla_g^{2} u \right)\vert_p \in\overline{\Gamma_{k}^{+}}.
\end{equation}

Then the Morse index of $u$ at $p$ is at most $n-k$.
\end{lem}

\begin{proof}

By Lemma~\ref{lem:spectral-convex}, the set
\[
\{A:\lambda(A)\in\Gamma_k^+\}
\]
is convex. Hence, by \ref{convex-condition},
\[
g^{-1}\nabla_g^2 u(p)
=
g^{-1}(S+\nabla_g^2 u)(p) + (-g^{-1}S(p))
\in \Gamma_k^+.
\]

Let the eigenvalues of $g^{-1} \nabla_g^{2}u(p)$ be ordered as $
\lambda_{1}\ge\lambda_{2}\ge\cdots\ge\lambda_{n}$.
By Lemma~\ref{CNSlemma},
\(
\lambda_{k}>0.
\)
It follows that $g^{-1}\nabla_g^{2}u(p)$ has at least $k$ positive eigenvalues,
and  at most $n-k$ negative eigenvalues.  This gives the stated bound on the Morse index.  
\end{proof} 
 
\begin{thm}
\label{thm:homology-vanishing-morse}
Let \(M\) be a smooth manifold, and let $f$
be a proper Morse function bounded from below. Suppose that every
critical point of \(f\) has Morse index at most \(N\). Then
\[
H_q(M;\mathbb Z)=0
\qquad
\text{for all } q>N .
\]
\end{thm}
 \begin{proof}
By the standard Morse-theoretic handle decomposition of sublevel sets
\cite[Part I, Section 3]{MilnorMorseTheory}, \(M\) has the homotopy type
of a CW complex with one cell of dimension \(m\) for each critical point
of Morse index \(m\). Since all critical indices are at most \(N\), this
CW complex has no cells in dimensions \(>N\), and therefore
\(H_q(M;\mathbb Z)=0\) for all \(q>N\).
\end{proof}

\begin{thm} \label{conThm}
Let $M^n$ be a compact connected oriented smooth manifold, possibly with
boundary, and let $\Sigma\subset M$ be a nonempty compact smooth embedded
submanifold of codimension at least one. If $\partial M\neq\emptyset$,
assume that $\partial M\subset \Sigma$.  Set $\Omega=M\setminus \Sigma.$
If
\[
H_{n-1}(\Omega;\mathbb Z)=0,
\]
then $\Sigma$ is connected.
\end{thm}

\begin{proof}
We first assume that $\partial M=\emptyset$. Write the connected components
of $\Sigma$ as
\begin{equation}
    \label{sum1}
\Sigma=\Sigma_1\sqcup\cdots\sqcup \Sigma_c .
\end{equation}
By the tubular
neighborhood theorem (
\cite[Chapter 4]{Hirsch} or \cite[Chapter 6]{LeeSmooth} ), there exist pairwise disjoint closed tubular neighborhoods
\(
N_i\subset M
\)
of the components $\Sigma_i$ such that each $N_i$ is a disk bundle of
positive rank over $\Sigma_i$. Set
\begin{equation}
    \label{sum2}N=\bigsqcup_{i=1}^c N_i .
\end{equation}
The radial projections in the normal
disk bundle give deformation retractions and  homotopy equivalences
\begin{equation} 
    \label{homotopy-equiv}
    N_i\setminus \Sigma_i \simeq \partial N_i .
\end{equation}

Since $\Sigma\subset \operatorname{int}N$, excision gives
\[
H_n(M,\Omega;\mathbb Z)
\cong
H_n(N,N\setminus \Sigma;\mathbb Z);
\]
see \cite[Section 2.1]{hatcher2002algebraic}. Using the decomposition \eqref{sum1} and \eqref{sum2},
we obtain
\begin{equation}\label{step1}
H_n(N,N\setminus \Sigma;\mathbb Z)
\cong
\bigoplus_{i=1}^c H_n(N_i,N_i\setminus \Sigma_i;\mathbb Z).
\end{equation}
It follows from \eqref{homotopy-equiv} that
\begin{equation}\label{step2}
H_n(N_i,N_i\setminus \Sigma_i;\mathbb Z)
\cong
H_n(N_i,\partial N_i;\mathbb Z).
\end{equation}
Since $N_i$ is a compact connected oriented $n$-manifold with boundary,
its relative fundamental class generates
\begin{equation}
    \label{step0}H_n(N_i,\partial N_i;\mathbb Z)\cong \mathbb Z.
\end{equation}
Equivalently, this is the top-dimensional case of
Poincar\'e--Lefschetz duality; see \cite[Section 3.3]{hatcher2002algebraic}. Therefore, from \eqref{step1} and \eqref{step2},
\begin{equation}\label{z-c}
    H_n(M,\Omega;\mathbb Z)
\cong
\mathbb Z^c .
\end{equation}

Now consider the long exact sequence of the pair $(M,\Omega)$:
\begin{equation}
    \label{exact}
H_n(M;\mathbb Z)
\longrightarrow
H_n(M,\Omega;\mathbb Z)
\longrightarrow
H_{n-1}(\Omega;\mathbb Z). 
\end{equation}
Since the last term equals zero by hypothesis,   the first map in \eqref{exact}
is surjective. Taking into account \eqref{step0} and \eqref{z-c}, this surjectivity is possible only if $c=1$. Thus, $\Sigma$ is connected in the closed
case.

We now consider the case $\partial M\neq\emptyset$. Let
\(
\widehat M
\)
denote the double of $M$ along $\partial M$, and let
\(
\widehat\Sigma\subset \widehat M
\)
denote the double of $\Sigma$. The collar neighborhood theorem again gives the
smooth structure on the double; see \cite[Chapter 9]{LeeSmooth}. Since
\(
\partial M\subset \Sigma,
\)
the complement of $\widehat\Sigma$ in $\widehat M$ is the disjoint union of
two copies of $\Omega$:
\[
\widehat M\setminus \widehat\Sigma
=
\Omega_+\sqcup \Omega_- .
\]
Therefore
\[
H_{n-1}(\widehat M\setminus \widehat\Sigma;\mathbb Z)
\cong
H_{n-1}(\Omega;\mathbb Z)
\oplus
H_{n-1}(\Omega;\mathbb Z)
=
0 .
\]
The manifold $\widehat M$ is closed, connected, and oriented, and
$\widehat\Sigma$ is a compact smooth embedded submanifold of codimension
at least one. Applying the closed case to the pair
\(
(\widehat M,\widehat\Sigma),
\)
we conclude that $\widehat\Sigma$ is connected. Therefore $\Sigma$ is connected.
\end{proof}

\bigskip 

We remark that the assumptions on $\Sigma$ in the previous theorem can be weakened considerably.   However, the precise statement is somewhat technical and will not be pursued here.

\section{Negatively Admissible Metrics and Topology} \label{AdmissibleSec}

Throughout this section, \((M^n,g)\), \(n\geq 3\), is a compact connected
smooth Riemannian manifold, possibly with boundary, and
\(\Sigma\subset M\) is a nonempty compact subset. If \(\partial M\neq
\emptyset\), we assume that \(\partial M\subset \Sigma\). We set
\[
\Omega:=M\setminus \Sigma .
\]
Thus \(\Omega\) is a smooth open manifold without boundary.

To conclude that \(\Sigma\) is connected, we shall later impose the
additional hypotheses of Theorem \ref{conThm}; namely, orientability of \(M\) and
the appropriate smooth embedded-submanifold assumption on \(\Sigma\).

Our first result is a Morse approximation lemma in the admissible setting:

\begin{lem} 
\label{lem:morse-approx-negative}
Suppose \(u\in C^{\infty}(\Omega)\) satisfies \vskip.1in 

$(i)$  $\lambda(g_u^{-1} A_u) \in \Gamma_k^{-}$ at each point in $\Omega$, and \vskip.1in 

$(ii)$ $u$ is  nondegenerate on $\Omega$.  \vskip.2in

Then, for every sufficiently large regular value \(\alpha\) of \(u\),
there exists a smooth function \(v\in C^\infty(\Omega)\) such that:
\begin{enumerate}
\item \(v=u\) on a neighborhood of \(\Omega\setminus\Omega_\alpha\), where
\[
\Omega_\alpha:=\{x\in\Omega:u(x)\le \alpha\};
\]
in particular,
\[
v(x)\to+\infty
\qquad\text{as } d_g(x,\Sigma)\to0 .
\]

\item \(v\) is a Morse function on \(\Omega\).

\item All critical points of \(v\) lie in the compact set \(\Omega_\alpha\).

\item For every \(\varepsilon>0\), the approximation may be chosen so that
\[
\|v-u\|_{C^2(\Omega_\alpha)}<\varepsilon .
\]

\item The negative \(k\)-admissibility condition is preserved: $\lambda(g_v^{-1} A_v)\in\Gamma_k^{-}(\Omega).$
\end{enumerate}
\end{lem}

\begin{proof}
Since \(u\) is  nondegenerate in the sense of Definition~\ref{nondegenerate}, $u$ has no critical points near $\Sigma$. Choose a sufficiently large
regular value \(\alpha\)  such that $$ \Omega\setminus\Omega_\alpha \subset\{d_g(x,\Sigma)<\epsilon_0\}.$$ Thus,
\(
du\neq0\) on $\Omega\setminus\Omega_\alpha$. Furthermore, there exists a collar
neighborhood \(U\) of \(\partial\Omega_\alpha\) and a constant \(c>0\) such
that
\[
|\nabla u|\ge c
\qquad\text{on }U.
\]
Choose open sets
\[
\partial\Omega_\alpha\subset U_1\Subset U_2\Subset U
\]
and a cutoff function \(\eta\in C^\infty(\Omega)\) satisfying
\[
0\le\eta\le1,\qquad
\eta\equiv0 \ \text{on } U_1\cup(\Omega\setminus\Omega_\alpha),
\qquad
\eta\equiv1 \ \text{on } \Omega_\alpha\setminus U_2.
\]
By Theorem~\ref{dense}, applied with \(K=\Omega_\alpha\) and \(m=2\),
for every \(\delta>0\) there exists a Morse function
\(
h\in C^\infty(\Omega)
\)
such that
\(
\|h-u\|_{C^2(\Omega_\alpha)}<\delta.
\)
Define
\[
v:=\eta h+(1-\eta)u
\qquad\text{on }\Omega.
\]
Since \(v-u=\eta(h-u)\), there exists a constant \(C_\eta>0\),
depending only on \(\eta\), such that
\[
\|v-u\|_{C^2(\Omega_\alpha)}
\le
C_\eta\,\|h-u\|_{C^2(\Omega_\alpha)}
<
C_\eta\delta.
\]
Choosing \(\delta<\varepsilon/C_\eta\), we obtain
\[
\|v-u\|_{C^2(\Omega_\alpha)}<\varepsilon.
\]

Since \(|\nabla u|\ge c\) on \(U\), choose \(\delta>0\) sufficiently
small so that \(v\) has no critical points in \(U\).
Since \(v=u\) on \(U_1\cup(\Omega\setminus\Omega_\alpha)\), and \(u\) has
no critical points there, every critical point of \(v\) lies in
\(
\Omega_\alpha\setminus U_2.
\)
On this set we have \(v=h\). Therefore every critical point of \(v\) is a
critical point of \(h\), and hence is nondegenerate because \(h\) is Morse.
Thus \(v\) is a Morse function on \(\Omega\), and all of its critical
points lie in the compact set \(\Omega_\alpha\).  It remains to check that the negative \(k\)-admissibility condition is
preserved. Since \(\Omega_\alpha\) is compact and
\[
\lambda(g_u^{-1}A_u)\in \Gamma_k^-
\qquad\text{on }\Omega_\alpha,
\]
the openness of \(\Gamma_k^-\), together with the continuity of the map
\[
w\longmapsto \lambda(g_w^{-1}A_w)
\]
in the \(C^2\)-topology, implies that, after choosing \(\delta>0\) smaller if
necessary,
\[
\lambda(g_v^{-1}A_v)\in \Gamma_k^-
\qquad\text{on }\Omega_\alpha.
\]
On \(\Omega\setminus\Omega_\alpha\), we have \(v=u\), and hence the same
condition holds there. Therefore
\[
\lambda(g_v^{-1}A_v)\in \Gamma_k^-
\qquad\text{on}\ \Omega.
\]
This proves the lemma.
\end{proof}

We now combine Lemma~\ref{lem:morse-approx-negative} with
Theorem~\ref{thm:homology-vanishing-morse}.

\begin{thm} 
\label{thm:vanishing-negative}
Assume \(\Omega=M\setminus\Sigma\) is connected, and that on $\Omega$, $\lambda(g^{-1}A_g)\in\overline{\Gamma_k^+}$. Suppose \(u\in C^{\infty}(\Omega)\) satisfies \vskip.1in 

$(i)$  $\lambda(g_u^{-1} A_u) \in \Gamma_k^{-}$ at each point of $\Omega$, and \vskip.1in 

$(ii)$ $u$ is  nondegenerate on $\Omega$. \vskip.1in 

Then
\[
H_q(\Omega;\mathbb Z)=0
\qquad
\text{for all } q>n-k .
\]
\end{thm}

\begin{proof}
By Lemma~\ref{lem:morse-approx-negative}, after replacing \(u\) by a
smooth function \(v\in C^\infty(\Omega)\), we may assume that \(v\) is a
Morse function on \(\Omega\) with \(v=u\) near \(\Sigma\).  Furthermore, $\lambda(g_v^{-1} A_v)\in\Gamma_k^-(\Omega)$, and all critical points of \(v\) lie in a compact subset of
\(\Omega\).  Since \(v=u\) near \(\Sigma\) and \(u\) is  nondegenerate, we have $v(x)\to+\infty$ as $d_g(x,\Sigma)\to0$, and $v$ is bounded from below.
It follows that \(v\) is a proper Morse exhaustion of \(\Omega\). In particular, its sublevel sets
\[
\Omega^b:=\{x\in\Omega:v(x)\le b\}
\]
are compactly contained in \(\Omega\) and exhaust \(\Omega\).

Let \(p\) be a critical point of \(v\). Since \(dv(p)=0\), the
conformal transformation formula for the Schouten tensor gives
\[
A_v(p)=A_g(p)-\nabla_g^2 v(p).
\]
Equivalently,
\[
A_v(p)+\nabla_g^2v(p)=A_g(p).
\]
We now compare the cone conditions using the background metric \(g\).
Since \(g_v=e^{2v}g\), we have
\[
g_v^{-1}A_v=e^{-2v}g^{-1}A_v.
\]
Because \(\Gamma_k^-\) is a cone and \(e^{-2v(p)}>0\), the condition
\(
\lambda(g_v^{-1}A_v)(p)\in \Gamma_k^-
\)
is equivalent to
\(
\lambda(g^{-1}A_v)(p)\in \Gamma_k^-.
\)
On the other hand, by   hypothesis we have
\[
\lambda(g^{-1}A_g)(p)\in \overline{\Gamma_k^+}.
\]
Applying Lemma~\ref{lem:morse-index-small} with \(S=A_v\), we have
\[
g^{-1}S\big|_p=g^{-1}A_v\big|_p\in\Gamma_k^-,
\]
and
\[
g^{-1}\bigl(S+\nabla_g^2v\bigr)\big|_p
=
g^{-1}A_g\big|_p
\in\overline{\Gamma_k^+}.
\]
Therefore the Morse index of \(v\) at \(p\) is at most \(n-k\).  We conclude that every critical point of the proper Morse function \(v\) has
Morse index at most \(n-k\). By
Theorem~\ref{thm:homology-vanishing-morse},
\[
H_q(\Omega;\mathbb Z)=0
\qquad
\text{for all } q>n-k .
\]
This proves the theorem.
\end{proof}

Note that Theorem \ref{Thm1} is an immediate corollary of Theorem \ref{thm:vanishing-negative}.  As another corollary, we have the following connectedness result under negative \(k\)-admissibility: 

\begin{cor}  
\label{cor:singular-set-connected-negative}   Under the assumptions of Theorem \ref{thm:vanishing-negative}, if $M$ is oriented, $\Sigma$ satisfies the condition of Theorem~\ref{conThm}, and
\(k\ge 2\), then \(\Sigma\) is connected.
\end{cor}

The proof follows from Theorems~\ref{thm:vanishing-negative} and~\ref{conThm}.

 
\end{document}